\documentclass[10pt]{amsart}
\usepackage{amsmath,amssymb,amsthm}
\newtheorem{thm}{Theorem}[section]
\newtheorem{lem}[thm]{Lemma}
\newtheorem{prop}[thm]{Proposition}

\numberwithin{equation}{section}

\theoremstyle{definition}

\theoremstyle{remark}
\newtheorem{remark}[thm]{Remark}

\newcommand{\Z}{{\mathbb Z}}
\newcommand{\R}{{\mathbb R}}
\newcommand{\T}{{\mathbb T}}

\def\11{{\rm 1~\hspace{-1.4ex}l} }

\title
[NLS on $\R^d\times \T$ ]
{Well-posedness and scattering for NLS on $\R^d\times \T $
in the energy space}
\author[Nikolay Tzvetkov]{Nikolay~Tzvetkov}
\author[Nicola Visciglia]{Nicola Visciglia}

\address{D\'epartement de Math\'ematiques, Universit\'e de Cergy-Pontoise, 2, 
avenue Adolphe Chauvin, 95302 Cergy-Pontoise  
Cedex, France and Institut Universitaire de France}\email{nikolay.tzvetkov@u-cergy.fr}
\address{Dipartimento di Matematica, Universit\`a Degli Studi di Pisa, 
Largo Bruno Pontecorvo 5 I - 56127 Pisa. Italy}\email{ viscigli@dm.unipi.it}

\begin{document}

\begin{abstract}
We study the Cauchy problem and the large data $H^1$ scattering for 
energy subcritical NLS posed on $\R^d\times \T$. \end{abstract}

\maketitle

\section{Introduction}
In our previous work \cite{TV}, we considered the nonlinear Schr\"odinger equation on a product space $\R^d\times M$, where $M$ is a compact riemannian manifold.
We have seen this problem as a kind of  vector valued nonlinear Schr\"odinger equation on $\R^d$ and we were able to get small data scattering results
(cf. also \cite{HPTV} for small data modified scattering results). Our goal here is to extend this view point to a large data problem in the very particular case when $M$ is the one dimensional torus. 

Therefore, our aim in this paper is the study of the local (and global) well-posedness
and scattering 
of the following family of Cauchy problems:
\begin{equation}\label{NLS}
\begin{cases}
i\partial_t u - \Delta_{x,y} u + u|u|^{\alpha}=0, \ \  (t,x,y) \in \R\times \R^d\times \T,\
\  d\geq 1\\
u(0,x,y)=f(x,y)\in H^1_{x,y},
\end{cases}
\end{equation}
where:
$$
\Delta_{x,y}=\sum_{i=1}^d \partial_{x_i}^2+ \partial^2_{y}.
$$ 
Concerning the Cauchy theory we shall assume $0<\alpha<4/(d-1)$ and for scattering we assume
$4/d<\alpha<4/(d-1)$. Our first result deals with the Cauchy problem.
\begin{thm}\label{halflwpStrong}
Let $d\geq 1$ and $0<\alpha<4/(d-1)$ be fixed. Then
we have:
\begin{enumerate}
\item \label{1bos}for any initial datum $f \in H^{1}
_{x,y},$ the problem \eqref{NLS} has a
unique local solution 
\begin{align*}
u(t,x,y)\in {\mathcal C}( (-T, T);H^{1}_{x,y}),
\end{align*}
where $T=T(\|f\|_{H^1_{x,y}})>0$;
\item \label{3bos}the solution $u(t,x,y)$ can be extended globally in time.
\end{enumerate}
\end{thm}
\begin{remark}
Property \eqref{3bos} follows by \eqref{1bos} due to the  
defocusing character of the the nonlinearity (a standard approximation argument is needed to justify the energy conservation).
Hence, along the paper, we focus mainly on the proof of \eqref{1bos}, i.e. the existence of an unique local solution
for any given initial datum. 
We also notice that the proof of \eqref{1bos} in Theorem~\ref{halflwpStrong} 
works also for the focusing NLS.
\end{remark}
The proof of the local existence given by 
Theorem \ref{halflwpStrong} goes as follows. First we prove the existence of one unique solution in the space
\begin{equation}\label{space}
L^q_tL^r_xH^{1/2+}_y\cap {\mathcal C}_t(H^1_{x,y}),
\end{equation}
where $(q,r)$ are Strichartz $\dot H^{1/2-}$-admissible for the propagator
$e^{it\Delta_x}$.
It is of importance for our analysis that a $H^{1/2}$ sub-critical nonlinearity in dimension $d$ is a $H^1$ sub-critical nonlinearity in dimension $d+1$. 
Therefore at the $x$ level we perform an $H^{1/2-}$ analysis and at the $y$ level, we perform the (trivial) $H^{1/2+}$ analysis which at the end enables us to perform an $H^1$ theory in the full sub-critical range of the nonlinearity. Incorporating in a non-trivial way the $y$ dispersive effect in this analysis is a challenging problem. Its solution may allow to extend our analysis to higher dimensional $y$ dependence. 
A key tool in order to perform a fixed point argument in the space \eqref{space} are 
the inhomogenous Strichartz estimates associated with $e^{it\Delta_x}$
(see \cite{F}, \cite{Vi}). The second step is the proof of the unconditional uniqueness
in the space ${\mathcal C}_t(H^1_{x,y})$.
We underline that the proof of Theorem \ref{halflwpStrong} in the range of nonlinearity $0<\alpha<4/d$, can be obtained following \cite{TTV}, where it is not needed the use of inhomogeneous Strichartz estimates for $e^{it\Delta_x}$.

In the cases $d=1$ for every $\alpha>0$ 
and $d=2,3$ for the $H^1$-critical nonlinearity $\alpha=4/(d-1)$, 
the proof of Theorem \ref{halflwpStrong}
can also be deduced respectively from the analysis in \cite{BGT1} and \cite{IP}.
The main point in our approach is that it works 
in $\R^d\times \T$ for every $d\geq 1$ and moreover it gives some
crucial controls of space-time global norms which are of importance for the scattering analysis.

The main result of this paper concerns the long-time behavior of the solutions given by Theorem~\ref{halflwpStrong}. 
\begin{thm}\label{main}
Assume $d\geq 1$ and $4/d<\alpha<4/(d-1)$, $f(x, y)\in H^1_{x,y}$ and 
$u(t, x,y)\in {\mathcal C}(\R;H^1_{x,y})$ be the unique global solution to \eqref{NLS}.
Then there exist $f_{\pm} \in H^1_{x,y}$
such that
$$\lim_{t\rightarrow \pm \infty} \|u(t,x,y) - e^{-it\Delta_{x,y}} f_{\pm}\|_{H^1_{x,y}}=0.$$
\end{thm}
\begin{remark}
Concerning scattering results for NLS in product spaces we quote \cite{HP}
where it is studied the quintic NLS on $\R\times \T^2$.
We also underline that
using the arguments of \cite{IP} (see also \cite{HP}) one may obtain that for $d=2,3$, the result of Theorem~\ref{main} also holds for the $H^1$ critical nonlinearity $\alpha=\frac{4}{d-1}$. One may also expect that these arguments provide an alternative (and more complicated) proof of Theorem~\ref{main} for $d=2,3$.
For $d\geq 4$, the extension of Theorem~\ref{main} to the  $H^1$ critical nonlinearity $\alpha=\frac{4}{d-1}$ is an open problem (even for the $H^1$ local theory). 

\end{remark}

\begin{remark}\label{thy} 
Notice that if one considers \eqref{NLS} on $\R^d\times \R$, then 
it is well-known that $H^1$-scattering is available for $4/(d+1)<\alpha<4/(d-1)$ (in contrast with 
Theorem~\ref{main} where we require the extra restriction $\alpha>4/d$). 
On the other hand the restriction $\alpha>4/d$ in Theorem~\ref{main}  is quite natural.
Indeed,  if we choose $f(x, y)=f(x)$ and $0<\alpha<4/d$
then the Cauchy problem \eqref{NLS} reduces
to $L^2$-subcritical NLS in $\R^d$, and at the best of our knowledge no
$H^1$-scattering result is available in this situation .
\end{remark}
%
It is well-known, since the very classical work \cite{GV}, that a key tool to prove scattering 
for NLS in the euclidean setting $\R^d$, with nonlinearities which are both energy subcritical 
and $L^2$-supercritical, is the proof of
the time-decay of the potential energy. 
In Proposition~\ref{decay} below we prove that this property persists  
for solutions to NLS on $\R^d\times \T$ in the energy subcritical regime (in particular
we do not need to require to the nonlinearity to be $L^2$-supercritical, see also Remark~\ref{subc} on this point).

A basic tool that we will use is a suitable version in the partially periodic setting
of the interaction Morawetz estimates,
first introduced in \cite{Iteam} to study the energy critical NLS in the euclidean space $\R^3$. Starting from this work the interaction Morawetz estimates have been exploited
in several other papers (\cite{CGT}, \cite{GVq}, \cite{PV}, \cite{tvz}, \cite{Vis}),
in particular they have been used to provide new and simpler proofs of the classical scattering results in \cite{GV} and \cite{Na}.

We emphasize that we make use of the interaction Morawetz estimates 
from a different point of view 
compared with the results above. In particular along the proof of Proposition
\ref{decay} below we are able to treat
in a unified and simple way NLS on $\R^d\times \T$ for every $d\geq 1$, without any distinction between
the cases $d\leq 3$ and $d>3$. 
This distinction is typical in previous papers involving interaction Morawetz estimates in the Euclidean setting (see Remark~\ref{gvd} for more details on this point).
Moreover it is unclear to us how to proceed, following the 
approach developed in previous papers  related with interaction Morawetz estimates, to prove Proposition~\ref{decay} in the case $d\geq 4$
(see Remark~\ref{car}).

Next we state the key proposition needed to prove Theorem~\ref{main}.

\begin{prop}\label{decay}
Let $u(t, x,y)\in {\mathcal C}(\R; H^1_{x,y})$ be a global solution to defocusing
NLS posed on $\R^d\times \T$ and with pure power nonlinearity $u|u|^{\alpha}$, 
with $
0<\alpha<4/(d-1)$. Then:
\begin{equation}\label{pas1}
\lim_{t\rightarrow \pm \infty} \|u(t, x,y)\|_{L^q_{x,y}}=0,\quad 2<q<\frac{2(d+1)}{d-1}.
\end{equation}
\end{prop}

\begin{remark}\label{subc}
Notice that in contrast with Theorem~\ref{main}, in Proposition~\ref{decay} we do not assume any lower-bound on $\alpha$.
Notice also that Proposition~\ref{decay} is not true for the focusing NLS 
on $\R^d\times \T$ for $\alpha<4/d$, even if the
initial data are assumed to be arbitrarly small in $H^1_{x,y}$. To prove this fact one can think about the solitary waves associated with the $L^2$ subcritical focusing NLS posed on $\R^d$, and notice that the corresponding $H^1_{x,y}$ norm can be arbitrary small.
\end{remark}
\begin{remark}\label{gvd} As already mentioned above
the proof of Proposition \ref{decay} is based on the use of interaction Morawetz 
estimates in the partially periodic setting.
Let us recall that the interaction Morawetz estimates 
allow to control the following quantity
(see for instance \cite{GVq}): 
\begin{equation}\label{GVqinter}
\int_{\R}\int_{\R^d} \big ||D_x|^\frac {3-d}2 (|u|^2) |^2 dx dt<\infty\end{equation}
for $u$ solution to NLS posed in the euclidean space $\R^d$.
Notice that via the Sobolev embedding it implies
some a -priori bounds of the type
$$\|u(t,x,y)\|_{L^p_tL^q_x}<\infty$$ in the case $d=1,2,3$.
This estimate is sufficient to deduce scattering on $\R^d$ for $d=1,2,3$
in the case of the nonlinearity $4/d<\alpha<4/(d-2)$.
Notice also that in higher dimensions $d\geq 4$ we get in \eqref{GVqinter}
the control of a negative derivative of $|u|^2$. In this case some extra work
is needed in order to retrieve the needed space-time summability that allows to get 
scattering. Typically 
the main strategy to overcome this difficulty 
is to retrieve some informations on negative derivative of $u$ via the following estimate (see \cite{tvz}):
\begin{equation}\label{vtz}
\||D_x|^\frac {3-d}4 f\|_{L^4}^2 \leq C \||D_x|^\frac {3-d}2 |f|^2\|_{L^2}.
\end{equation}
Once a negative derivative of $u$ is estimated, then we can be interpolated with the bound $\|u\|_{L^\infty_t H^1_x}$,
and hence we get the needed space-time integrality necessary to prove scattering
for $d\geq 4$.
\end{remark}
\begin{remark}\label{car}
We underline that arguing as in \cite{ACD}, where it is studied NLS with a partially confining potential, one can prove the following version of interaction Morawetz estimate:
\begin{equation}\label{interpart}\int_\R \int_{\R^d} \big ||D_x|^\frac {3-d}2 (\int_{\T} |u(t,x,y)|^2 dy)\big |^2 dxdt<\infty
\end{equation}
provided that $u(t,x,y)$ solves NLS posed on $\R^d\times \T$.
Hence via the Sobolev embedding one can deduce 
some a -priori bounds
$$\|u(t,x,y)\|_{L^p_tL^q_xL^2_y}<\infty$$ in the case $d=1,2,3$. 
This estimate is sufficient to deduce scattering for $4/d<\alpha<4/(d-1)$
and $d=1,2,3$ (see the computations
in \cite{ACD} in the case of a partially confining potential).
However, as far as we can see, it is unclear how to exploit \eqref{interpart} in the case $d\geq 4$. 
\end{remark}

\begin{remark}
Estimate \eqref{interpart} is obtained by controlling 
a suitable family of multiple integrals of the type:
$$\int_\R\int_{\R^d}\int_{\R^d}\int_{\T}\int_{\T} ..... dx_1dx_2dy_1 dy_2 dt,$$
where the integrand function depends on a test function $\varphi$ 
and on the solution $u$ to NLS . Once this test function
is suitably  choosen then it allows to contract the variables $x_1$, $x_2$, $y_1, y_2$ 
to $x,y$, hence we get \eqref{interpart}.
The main point in our analysis is that we combine an argument by the absurd in conjunction
with  the finiteness of the following quantity 
\begin{align}\label{inter167}
\int_{\R} &\big ( \sup_{x_0\in \R^d} \int\int_{Q^d(x_0,r)\times (0, 2\pi)}
|u(t, x, y)|^{2}dxdy\big )^\frac{\alpha+4}2 dt<\infty
\end{align}
that in turn follows by
\begin{equation}\label{inter168}\int_\R \int_{\R^d}\int_{\R^d}\int_{\T}\int_{\T} 
\Delta_x \varphi (x_1-x_2)|u(x_1, y_1)|^{\alpha+2} |u(x_2, y_2)|^2 dt dx_1dy_1 dx_2dy_2<\infty,
\end{equation}
where $\varphi$ is any convex function. In this estimate we choose
$\varphi=\langle x \rangle$. Notice that this choice
does not allows contraction of the variables
$(x_1, x_2, y_1, y_2)$ in $(x,y)$,
however it implies \eqref{inter167}, which is sufficient to conclude the time 
decay of the potential energy for solutions to NLS in a simpler way 
and in a more general setting compared with \eqref{interpart}. 
We believe that this part of our argument is of independent interest.  
\end{remark}
\section{Some Useful Functional Inequalities}\label{foschi}
In this section we collect some a-priori estimates 
for the propagator $e^{ -i t \Delta_{x,y}}$
and the associated Duhamel operator.
At the end we also present an anisotropic 
Gagliardo-Nirenberg inequality that will be useful in the sequel.\\
We define $H^{s}_{x}H^{\gamma}_{y}$ as
$$
H^{s}_{x}H^{\gamma}_{y}=(1-\Delta_x)^{-\frac s2} (1-\partial^2_y)^{-\frac \gamma2}L^2_{x,y},
$$
endowed with the natural norm.
\begin{prop}\label{StriProd}
Let $\gamma\in\R,$ $0\leq s<d/2$ and $d\geq 1$. Then we have the following homogeneous estimates:
\begin{equation}
\label{eq:200p}
 \|e^{- i t \Delta_{x,y}} f\|_{L^q_t L^r_xH_y^{\gamma}}\leq C\|f\|_{ H_x^{s}H_y^{\gamma}},
 \end{equation}
provided that the following conditions hold:
 \begin{equation}
\label{eq:201p}  \frac{2}{q}+\frac{d}{r}= \frac{d}{2}-s, \quad \quad q\geq 2, 
\quad \quad (q,d)\neq (2,2).
\end{equation}
\end{prop}
\begin{proof} We claim that by combining the Sobolev embedding 
with the usual Strichartz estimates on $\R^d$ we get
\begin{equation}
\label{starlit}
 \|e^{ -i t \Delta_{x}} h\|_{L^q_t L^r_x}\leq C\|h\|_{ H_x^{s}},
 \end{equation}
where $q,r,s$ are as in the assumptions. 
By the same argument as in \cite{TV} the estimate \eqref{starlit}
implies $$\|e^{ -i t \Delta_{x,y}} f\|_{L^q_t L^r_xL^2_y}\leq C\|f\|_{ H_x^{s}L^2_y}.$$
We can conclude by using the fact that $(\sqrt {1-\partial_y^2})^\gamma$
commute with the linear Schr\"odinger equation on $\R^d\times \T$.\\
Next we give a few details about the proof of \eqref{starlit}.
Given any $q\geq 2$ for $d\geq 3$ (or $q>2$ for $d=2$, $q\geq 4$ for $d=1$) we fix the unique $2\leq \tilde r<\infty$ such that
\begin{equation}\label{confgh}\frac 2q+ \frac d{\tilde r}=\frac d2.\end{equation}
Hence by the usual Strichartz estimates (see \cite{KT}) we get:
\begin{equation*}
 \|e^{ -i t \Delta_{x}} h\|_{L^q_t L^{\tilde r}_x}\leq C\|h\|_{L^2_x}.
 \end{equation*}
In turn it implies
\begin{equation*}
 \|e^{ -i t \Delta_{x}} h\|_{L^q_t W^{s,\tilde r}_x}\leq C\|h\|_{H^s_x}.
 \end{equation*} 
Notice that if $s\cdot \tilde r<d$ then we conclude by the sharp Sobolev embedding
$W^{s, \tilde r}_x\subset L^r_x$ (here $r$ is precisely the one that appears in
\eqref{eq:201p} once $q$ and $s$ are fixed).
In the case $s \cdot \tilde r\geq d$ we conclude again by the Sobolev embedding 
$W^{s, \tilde r}_x\subset L^p_x$ for every $\tilde r\leq p<\infty$, and 
in particular $W^{s, \tilde r}_x\subset L^r_x$.
 
\end{proof}

\begin{prop}\label{StriProd1}
Let  $\gamma\in \R$ and $d\geq 1$. Indicate by $D$ both $\partial_{x_j}$, $j=1,..,d$ and $\partial_y.$ 
Then we have for $k=0,1$ the following estimates:
\begin{align}
\label{eq:200p1}
 \|D^k e^{ -i t \Delta_{x,y}} f\|_{L^l_t L^p_xH^\gamma_y} + \left \|D^k \int_0^t e^{- i (t-\tau) \Delta_{x,y}} F(\tau)
d\tau\right\|_{L^l_t L^p_x H^\gamma_y}&\\\nonumber
\leq C\big (\|D^k f\|_{ L^2_xH^\gamma_y}+\|D^kF\|_{L^{\tilde l'}_t
L^{\tilde p'}_xH^\gamma_y}\big ),&
 \end{align}
 provided that
 \begin{equation*}
 \frac{2}{l}+\frac{d}{p}=\frac{2}{\tilde l}+\frac{d}{\tilde p}= \frac{d}{2}, \quad \quad 
 l\geq 2, \quad (l,2)\neq (2,2).
 \end{equation*}
\end{prop}
\begin{proof}
The proof follows by the Strichartz estimates 
associated with the propagator
$e^{-it\Delta_x}$, in conjunction with the argument in \cite{TV}. 
\end{proof}

\begin{prop}\label{StriProd2}
Let $\gamma\in \R$ be fixed and $d\geq 3$. Then we have the following extended inhomogeneous estimates:
\begin{equation}
 \label{eq:202p}
 \left \|\int_0^t e^{- i (t-\tau) \Delta_{x,y}} F(\tau)
d\tau\right\|_{L^q_t L^r_xH_y^{\gamma}}\leq C\|F\|_{L^{\tilde{q}'}_t
L^{\tilde{r}'}_xH_y^{\gamma}}
\end{equation}
provided that:
\begin{equation}\label{nodfos}
0<\frac 1q,\frac 1r, \frac 1{\tilde q}, \frac 1{\tilde r}<\frac 12
\end{equation}
\begin{align}
\label{eq.ass3v3ca}
\frac 1 q+ \frac 1{\tilde q}<1, \quad \quad 
& \frac{d-2}{d}<\frac{r}{\tilde{r}}< \frac{d}{d-2}\\
\label{martfriben}
\frac{1}{q}+\frac{d}{r}<\frac{d}{2}, \quad \quad & 
\frac{1}{\tilde{q}}+\frac{d}{\tilde{r}}<\frac{d}{2},\quad \quad\frac{2}{q} + \frac dr + \frac{2}{\tilde{q}} +\frac d{\tilde r}=d.
\end{align} 
The same conclusion holds for $d=1,2$ provided that we drop the conditions
\eqref{eq.ass3v3ca}.
\end{prop}
\begin{proof}
In the case that $f$ and $F$ do not depend on $y$, the estimates above are special cases of the inhomogeneous extended Strichartz estimates proved in Thm. 1.4 \cite{F}
(see also \cite{Vi}).
Its extension to the case that we have explicit dependence on $y$ (in $f$ and/or $F$)
follows arguing as in \cite{TV}.
We underline that in order to apply the technique in \cite{TV}
we need \eqref{nodfos},  which is not required in \cite{F}.
\end{proof}
\begin{remark}
The interest of using the estimates of \cite{F} is that it allows to avoid to differentiate at a fractional order the nonlinearity $|u|^\alpha u$ with respect to the $x$ variable.
Therefore the only fractional Leibniz rule we need is Lemma~\ref{TzBil} below.
\end{remark}
The following result will be useful in the sequel.
\begin{lem}\label{inter}
Let $d\geq 1$ and $u_n(x,y)$ be a sequence such that $\|u_n\|_{H^1_{x,y}}=O(1)$, $\|u_n\|_{L^{p}_{x,y}}=o(1)$
for some $2<p<\infty$. Then for every $2<r<\frac{2d}{d-1}$ there exists $\delta>0$ such that
$\|u_n\|_{L^r_{x} H^{1/2+\delta}_y}=o(1)$.
\end{lem}

\begin{proof}
By combining the assumption with the Sobolev embedding we obtain that 
\begin{align}\label{allq}\|u_n\|_{L^{q}_{x,y}}=o(1) \quad & \forall \,\,2<q<\infty
\hbox{ for } d=1,2, \\\nonumber &\forall\,\,
2<q<\frac{2d}{d-2} \hbox{ for } d\geq 3.\end{align}
First we prove the following estimate that will be useful in the sequel:
\begin{equation}\label{lysi}\forall \gamma, \quad \quad \exists C=C(\gamma)>0 \hbox{ s.t. }  
\|v\|_{L^{2d/(d-1)}_x H^{1/2-\gamma}_y}\leq C \|v\|_{H^1_{x,y}}.\end{equation}
To prove this estimate we develop $v(x,y)$ in Fourier series w.r.t. y variable:
$$v(x,y)= \sum_{n\in \Z} v_n(x) e^{iny}.$$
Hence by the Minkowski inequality we get
$$\|v\|_{L^{2d/(d-1)}_x H^{1/2-\gamma}_y}^2 
=\|\sum_{n\in \Z} \langle n\rangle^{1-2\gamma} |v_n(x)|^2\|_{L^{d/(d-1)}_x}
\leq \sum_{n\in \Z} \langle n\rangle^{1-2\gamma} \|v_n(x)\|_{L^{2d/(d-1)}_x}^2$$
and by the Haussdorf-Young inequality
$$...\leq C \sum_{n\in \Z} \langle n\rangle^{1-2\gamma} \|\hat v_n(\xi)\|_{L^{2d/(d+1)}_x}^2$$
$$\leq C \sum_{n\in \Z} \langle n\rangle^{1-2\gamma} 
\big (\int |\hat v_n(\xi)|^2\langle \xi\rangle^{1+2\gamma} d\xi\big)\cdot\big (\int \langle \xi\rangle^{-d-2\gamma d} d\xi \big)^{1/d}
\leq C \|v\|_{H^1_{x,y}}^2.$$
Next, we shall prove
\begin{equation}\label{rexact}
\exists\,\, 2<r_0<\frac{2d}{d-1} \hbox{ s.t. } 
\|u_n\|_{L^{r_0}_x H_y^\frac{4d-1}{4d}}=o(1).
\end{equation}
Once \eqref{rexact} is proved then we conclude by interpolation
between \eqref{rexact} and \eqref{lysi} in the case $r_0\leq r<\frac{2d}{d-1}$.
In the case $2<r<r_0$ then we can interpolate between \eqref{rexact} and the estimate
$\|u_n\|_{L^2_x H^1_y}=O(1)$ (that follows by the assumptions).
Next we focus on \eqref{rexact}. Notice that we have
the following Gagliardo-Nirenberg inequality:
$$\|v(x,.)\|_{H^{s_0}_y}\leq C\|v(x,.)\|_{L^2_y}^{1-s_0}\|v(x,.)\|_{H^1_y}^{s_0}$$
where we have fixed $s_0=\frac{4d-1}{4d}$.
In turn by the H\"older inequality it gives:
\begin{equation*}\big \|\|v(x,.)\|_{H^{s_0}_y}\big\|_{L^{r_0}_x} \leq C \big\|\|v(x,.)\|_{L^2_y}^{1-s_0}\big\|_{L^{p_0}_x}
\big\|\|v(x,.)\|_{H^1_y}^{s_0}\big\|_{L^{2/s_0}_x},\end{equation*}
where $$\frac{1}{r_0}=\frac 1{p_0} + \frac {s_0}2$$
and $p_0=8(d+1)$.
Since $(1-s_0)p_0>2$ we can use  
the trivial estimate $\|v(.,y)\|_{L^2_y}\leq  \|v(.,y)\|_{L^{(1-s_0)p_0}_y}$ and we get
\begin{align*}\big \| \|v(x,.)\|_{H^{s_0}_y}\big\|_{L^{r_0}_x} & \leq C 
\|v(x,.)\|_{L^{(1-s_0)p_0}_{x,y}}^{1-s_0}
\big\|\|v(x,.)\|_{H^1_y}\big \|_{L^{2}_x}^{s_0}
\\\nonumber &\leq C  
\|v(x,.)\|_{L^{(1-s_0)p_0}_{x,y}}^{1-s_0} \|v\|_{H^1_{x,y}}^{s_0}
.\end{align*}
Since $2<(1-s_0)p_0<\frac{2d}{d-2}$ for $d\geq 3$ and $2<(1-s_0)p_0<\infty$
for $d=1,2$ we conclude by \eqref{allq}.
\end{proof}
\section{Fixing the admissible exponents for the well-posedness analysis}
We collect in this section some preparations, useful in the sequel to construct
suitable functional spaces in which we shall perform a contraction argument 
to guarantee existence and uniqueness of solutions to \eqref{NLS}.\\
The next proposition will be useful to study the Cauchy problem associated with 
\eqref{NLS} in the regime
$0<\alpha<4/d$.
\begin{prop}\label{strichclass}
Let $d\geq 1$ and $0<\alpha<4/d$ be fixed. Then there exist
$(q, r)\in [2, \infty]\times [2,\infty]$ such that
\begin{align*}\frac 2q + \frac dr&=\frac d2, \quad \quad (q, d)\neq (2, 2)\\\nonumber
\frac 1{q'}>\frac{\alpha+1}q&, \quad \quad \frac 1{r'}=\frac{\alpha+1}r.\end{align*}
\end{prop}
\begin{proof}
Choose $(\frac 1q,\frac 1r)=(\frac{d\alpha}{4(\alpha+2)} , \frac 1{\alpha+2})$. 

\end{proof}
To study the Cauchy problem \eqref{NLS} in the regime $4/d\leq \alpha<4/(d-1)$
we shall need the following Proposition.
 \begin{prop}\label{AlgebricCauchyfer}
Let $d\geq 3$ and $4/d\leq \alpha<4/(d-1)$ be fixed.
Then there exists $0< s<1/2$ and
$(q, r, \tilde{q}, \tilde{r})$
such that:
\begin{equation}\label{easyconitionv3fer}
0<\frac 1q,\frac 1r, \frac 1{\tilde q}, \frac 1{\tilde r}<\frac 12
\end{equation} and
\begin{align}
\label{eq.ass3v3fer}\frac 1 q+ \frac 1{\tilde q}<1, \quad \quad 
& \frac{d-2}{d}<\frac{r}{\tilde{r}}< \frac{d}{d-2}\\
\label{dropd12fer}
\frac{1}{q}+\frac{d}{r}<\frac{d}{2}, \quad \quad & 
\frac{1}{\tilde{q}}+\frac{d}{\tilde{r}}<\frac{d}{2}\\
\label{eq.ass102v3fer}  \frac{2}{q} +\frac{d}{r}=\frac d2- s,\quad \quad &
\frac{2}{q} + \frac dr + \frac{2}{\tilde{q}} +\frac d{\tilde r}=d,\\
\label{eq.ass9bisvfer} \frac{1}{\tilde{q}'}>\frac{\alpha+1}{q}, \quad \quad& \frac{1}{\tilde{r}'}=\frac{\alpha+1}{r}.&
\end{align}
For $d=1,2$ we get the same conclusion, provided that we drop 
conditions \eqref{eq.ass3v3fer}.
Moreover we can assume
\begin{equation}\label{impBosUl3fer}
\frac{\alpha}q+\frac{\alpha d}{2r}<1,\quad \quad \ \frac \alpha r<1.
\end{equation}
\end{prop}
We shall need the following Lemma.
\begin{lem}\label{AlgebricCauchy}
Let $d\geq 3$, $4/d\leq \alpha<4/(d-1)$ be fixed and
$s=\frac{\alpha d -4}{2\alpha}$.
Then there exist
$(q, r, \tilde{q}, \tilde{r}) $
such that:
\begin{equation}\label{easyconitionv3}0<\frac 1q,\frac 1r, \frac 1{\tilde q}, \frac 1{\tilde r}<\frac 12
\end{equation} and
\begin{align}
\label{eq.ass3v3}\frac 1 q+ \frac 1{\tilde q}<1, \quad \quad 
& \frac{d-2}{d}<\frac{r}{\tilde{r}}< \frac{d}{d-2}\\
\label{dropd12}
\frac{1}{q}+\frac{d}{r}<\frac{d}{2}, \quad \quad & 
\frac{1}{\tilde{q}}+\frac{d}{\tilde{r}}<\frac{d}{2}\\
\label{eq.ass102v3}  \frac{2}{q} +\frac{d}{r}=\frac d2- s,\quad \quad &
\frac{2}{q} + \frac dr + \frac{2}{\tilde{q}} +\frac d{\tilde r}=d,\\
\label{eq.ass9bisv} \frac{1}{\tilde{q}'}=\frac{\alpha+1}{q}, \quad \quad& \frac{1}{\tilde{r}'}=\frac{\alpha+1}{r}.&
\end{align}
For $d=1,2$ we get the same conclusion, provided that we drop 
conditions \eqref{eq.ass3v3}.\\
Moreover we can also assume that
\begin{equation}\label{impBosUl3_gif}
\frac{\alpha}q+\frac{\alpha d}{2r}=1,\quad \quad \ \frac \alpha r<1.
\end{equation}
\end{lem} 
\begin{remark}
Compared with Proposition \ref{AlgebricCauchyfer}, in Lemma
we have fixed $s$ and moreover we put identity
in \eqref{eq.ass9bisv} (compare with \eqref{eq.ass9bisvfer} where we have inequality).
\end{remark}
\begin{proof}
First we show
that by our choice of $s$ \eqref{impBosUl3_gif} follows. Indeed we get
$$\frac{\alpha}{q}+\frac{\alpha d}{2r}=\frac{\alpha}2 \left(\frac{2}{q}+\frac dr\right)=\frac{\alpha}2\left(\frac{d}{2}-s\right)=1$$
where we used
\eqref{eq.ass102v3}. Notice also that by the second identity in \eqref{eq.ass9bisv}
we get $\frac \alpha r<1$, in fact 
$\frac \alpha r<\frac{\alpha+1}r + \frac 1{\tilde r}=1$.\\
Moreover the first condition in \eqref{eq.ass3v3} follows by \eqref{easyconitionv3},
and the first condition in \eqref{dropd12} follows by the first identity in \eqref{eq.ass102v3}. Hence since now on we can skip those conditions. 
It is easy to
check that thanks to our choice of $s$, the identities in \eqref{eq.ass102v3} and \eqref{eq.ass9bisv} are not independent. Moreover by \eqref{eq.ass102v3} and \eqref{eq.ass9bisv}, and by recalling
$s=\frac{\alpha d -4}{2\alpha}$,
we get 
\begin{equation}\label{condexpl}\frac 1{\tilde q}=-\frac 1\alpha + \frac{(\alpha+1)d}{2r},\quad \quad
\frac 1{\tilde r}= 1 -\frac{\alpha+1}r, \quad \quad \frac 1q= \frac 1\alpha - \frac d{2r}.
\end{equation}
Next we consider two cases:
\\
\\
{\em First case: $d\geq 3$}
\\
\\
Thanks to \eqref{condexpl}, the conditions \eqref{easyconitionv3}, \eqref{eq.ass3v3}
(where we skip the first one), \eqref{dropd12} (where we skip the first one)
can be written as follows:
\begin{align*} &\frac{\alpha d}{2}<r<\frac{\alpha d }{(2-\alpha)}, 
r>2, \frac{\alpha(\alpha+1) d}{\alpha+2}<r< \frac{\alpha(\alpha+1) d}{2}, 
\\\nonumber 
&\alpha+1<r<2(\alpha+1), 
r<\frac{\alpha (\alpha+1) d}{\alpha d -2}, 
\frac {d-2}{d}+\alpha+1< r< \frac{d}{d-2}+\alpha+1\end{align*}
Hence we conclude tat we can select a suitable $r$ if the condition
\begin{align*}\max\{\frac{\alpha d}{2}, &2, \frac{\alpha(\alpha+1) d}{\alpha+2}, \alpha+1,
\frac {d-2}{d}+\alpha+1 \}\\\nonumber&<\min\{ \frac{\alpha d }{(2-\alpha)},
\frac{\alpha(\alpha+1) d}{2}, 2(\alpha+1), \frac{\alpha  (\alpha+1)d}{\alpha d -2},
\frac{d}{d-2}+\alpha+1\}
\end{align*}
is satisfied. Since we are assuming $4/d\leq \alpha<4/(d-1)$ this condition is equivalent to
\begin{align}\label{LYTH}\max\{\frac{\alpha(\alpha+1) d}{\alpha+2},
&\frac {d-2}{d}+\alpha+1 \}\\\nonumber
&<\min\{ \frac{\alpha d }{(2-\alpha)}, \frac{\alpha (\alpha+1) d}{\alpha d -2},
\frac{d}{d-2}+\alpha+1\}.\end{align}
Next we notice that \begin{equation}\label{LYSTH}\frac{\alpha(\alpha+1) d}{\alpha+2}=\max\{\frac{\alpha(\alpha+1) d}{\alpha+2},
\frac {d-2}{d}+\alpha+1 \}.\end{equation}
In fact it follows by direct computation for $d=3,4$ and for $d\geq 5$ it comes by the following argument.
Notice that $\frac{\alpha(\alpha+1)d}{\alpha+2}\geq \frac{d-2}{d} +(\alpha+1)$
is equivalent to
$ (\alpha+1) \big( \frac{\alpha d}{\alpha+2} -1 \big)\geq \frac{d-2}{d}$, that under the constrain
$4/d\leq \alpha<4/(d-1)$ can be written as
$$\inf_{\alpha\in [4/d, 4/(d-1))} (\alpha+1) \big( \frac{\alpha d}{\alpha+2} -1 \big)\geq \frac{d-2}{d}.$$
In turn this inequality follows provided that
$(1+4/d) \big( \frac{d\cdot 4/d }{4/(d-1)+2} -1 \big)\geq \frac{d-2}{d}$,
and by elementary computations it is equivalent to
$(d+4)(d-3)\geq (d+1)(d-2)$,
which is satisfied for every
$d\geq 5$.
Hence by \eqref{LYTH} and \eqref{LYSTH} we conclude provided that we show
$$\frac{\alpha(\alpha+1) d}{\alpha+2}<
\frac{\alpha d }{(2-\alpha)}, \frac{\alpha(\alpha+1) d}{\alpha+2}<
\frac{\alpha  (\alpha+1) d}{\alpha d -2},
\frac{\alpha(\alpha+1) d}{\alpha+2}< \frac{d}{d-2}+\alpha+1.$$
The first and second inequalities are satisfied for any $0<\alpha<4/(d-1)$
and the last one follows by
\begin{equation}\label{equiv}\sup_{\alpha\in [4/d, 4/(d-1))} (\alpha+1) \big( \frac{\alpha d}{\alpha+2} -1 \big)< \frac{d}{d-2}.\end{equation}
On the other hand we have
$$\sup_{\alpha\in [4/d, 4/(d-1))} (\alpha+1) \big( \frac{\alpha d}{\alpha+2} -1 \big)
\leq (1+4/(d-1)) \big( \frac{4d/(d-1) }{4/d+2} -1 \big)$$
$$= \frac{(d+3)(d^2-d+2)}{(d-1)^2(d+2)}.$$
Hence \eqref{equiv} follows provided that
$\frac{(d+3)(d^2-d+2)}{(d-1)^2(d+2)}< \frac{d}{d-2}$,
which is always satisfied.\\
\\
\\
{\em Second case: $d=1,2$}
\\
\\
Arguing as above (recall that we drop \eqref{eq.ass3v3}) we conclude provided that we can select $r$ such that:
\begin{align*} &r>\frac{\alpha d}{2}, 
r>2, \frac{\alpha(\alpha+1) d}{\alpha+2}<r< \frac{\alpha(\alpha+1) d}{2}, 
\\\nonumber 
&\alpha+1<r<2(\alpha+1), 
r<\frac{\alpha  (\alpha+1)d}{\alpha d -2}.\end{align*}
By elementary computations (see the case $d\geq 3$) and by recalling $4/d\leq \alpha<4/(d-1)$,
the conditions above are equivalent to the following inequality:
\begin{equation}\label{TvZX}\max \{(\alpha+1),\frac{\alpha(\alpha+1) d}{\alpha+2}\}
<r<\frac{\alpha (\alpha+1) d}{\alpha d -2}.\end{equation}
On the other hands by explicit computation we get
$$\max \{(\alpha+1),\frac{\alpha(\alpha+1) d}{\alpha+2}\}=
\frac{\alpha(\alpha+1) 2}{\alpha+2} \hbox{ for } d=2,$$$$
\max \{(\alpha+1),\frac{\alpha(\alpha+1) d}{\alpha+2}\}=
\alpha+1 \hbox{ for } d=1,$$
and \eqref{TvZX} follows by elementary considerations.
 
\end{proof}

\noindent{\bf Proof of Proposition \ref{AlgebricCauchyfer}.}
\\
\\
We focus on the case $d\geq 3$ (the cases $d=1,2$ can be treated by a similar argument).
We argue by a continuity argument based on
Lemma \ref{AlgebricCauchy}.
In fact we fix $(\frac 1q, \frac 1r, \frac 1{\tilde q}, \frac 1{\tilde r}, s)$
as in Lemma \ref{AlgebricCauchy} and we look for 
$(\frac 1{q+\epsilon}, \frac 1r, \frac 1{{\tilde q}_\epsilon}, \frac 1{\tilde r}, s_\epsilon)$
that satisfy conditions of Lemma \ref{AlgebricCauchyfer}, for some $\epsilon>0$ small enough
and $\tilde q_\epsilon, s_\epsilon$ will be properly chosen in dependence of $\epsilon$. 
By our choice it will be clear that $\lim_{\epsilon\rightarrow 0}
s_\epsilon=s$ and $\lim_{\epsilon\rightarrow 0}
\tilde q_\epsilon=\tilde q$.
Notice that with this choice the identity in \eqref{eq.ass9bisvfer} 
is satisfied (compare with \eqref{eq.ass9bisv}). Also
\eqref{easyconitionv3fer}, \eqref{eq.ass3v3fer}, \eqref{dropd12fer}
are satisfied by a continuity argument provided that $\epsilon>0$ is small enough
(recall that $r, \tilde r, q, \tilde q$ satisfy
\eqref{easyconitionv3}, \eqref{eq.ass3v3}, \eqref{dropd12}).
Notice also that since $q+\epsilon>q$ then the first identity in \eqref{eq.ass102v3fer} is satisfied provided that we choose $s_\epsilon>s$ (recall 
that $q,r,s$ satisfy the first identity in \eqref{eq.ass102v3}) and also
\eqref{impBosUl3fer} follows by \eqref{impBosUl3}.
\\
Next we impose that $\frac 1{q+\epsilon}, \frac 1r, \frac 1{{\tilde q}_\epsilon}, \frac 1{\tilde r}$
satisfy the second identity in \eqref{eq.ass102v3fer}, i.e.
$\frac{1}{q+\epsilon} +\frac{1}{\tilde{q}_\epsilon}=\frac d2 (1-\frac 1r -\frac 1{\tilde r})=\beta$.
We claim that 
$\frac{\alpha+1}{q+\epsilon}+\frac 1{\tilde q_\epsilon}<1$ (notice this is equivalent
to the inequality in \eqref{eq.ass9bisvfer})
and it will conclude the proof.
Indeed we write
$\frac 1{q+\epsilon} = \frac 1 q - \frac{\epsilon}{q(q+\epsilon)}$
and hence $\frac{1}{\tilde{q}_\epsilon}= \beta - \frac 1q + \frac{\epsilon}{q(q+\epsilon)}
=\frac 1{\tilde q} + \frac{\epsilon}{q(q+\epsilon)} $,
where we used 
$\frac{1}{q}+\frac 1{\tilde q}=\beta$ (see the second identity in \eqref{eq.ass102v3}). Hence
we get by the first identity in \eqref{eq.ass9bisv}
$$\frac{\alpha+1}{q+\epsilon}+\frac 1{\tilde q_\epsilon}
= \frac{\alpha+1}{q} - \frac{\epsilon(\alpha+1)}{q(q+\epsilon)}
+\frac 1{\tilde q} + \frac{\epsilon}{q(q+\epsilon)}
= - \frac{\epsilon \alpha}{q(q+\epsilon)} +1<1.
$$

\hfill$\Box$

\begin{lem}\label{sTrbis}
Let $d\geq 1$ and $4/d\leq \alpha<4/(d-1)$ be fixed.
Then there exist $2< l\leq \infty, 2\leq p\leq \infty$
such that:
\begin{align}
\label{st1bis}&\frac 2l + \frac dp=\frac d2,\\
\label{st2bis}&\frac 1{p'}=\frac 1p+\frac \alpha r,\\
\label{st3bis}&\frac 1{l'}>\frac 1l+ \frac \alpha q,
\end{align}
where $(q,r)$ is any couple given by Proposition
\ref{AlgebricCauchyfer}.
\end{lem}

\begin{proof}
The conditions \eqref{st1bis} and \eqref{st2bis} imply
$$\frac 1l= \frac {\alpha d}{4r},\quad \quad \frac 1p= \frac 12 - \frac \alpha{2r}$$
and hence the condition 
\eqref{st3bis} becomes: \begin{equation}\label{usef}\frac{\alpha d}{2r} +\frac{\alpha}{q}<1,\end{equation}
which is verified by \eqref{impBosUl3fer}.
The last condition to be checked is that if $l, p$ are as above then
$l,p>0$. Indeed $l>0$ is trivial and $p>0$ is equivalent to
$\frac \alpha r<1$ which follows by  \eqref{impBosUl3fer}.
 
\end{proof}

\section{Proof of Theorem \ref{halflwpStrong}}
Along this section we need the following lemma to treat the nonlinear term.
\begin{lem}\label{TzBil}
For every $0<s<1, \alpha>0$ there exists $C=C(\alpha, s)>0$ such that:
$$\|u|u|^\alpha\|_{\dot H^s_y}\leq C \|u\|_{\dot H^s_y} \|u\|_{L^\infty}^\alpha.$$
\end{lem}
\begin{proof}
First we prove the following identity:
\begin{equation}\label{equivHs.}
\int_0^{2\pi}\int_\R \frac{|u(x+h)-u(x)|^2}{|h|^{1+2s}} dxdh
= c \|u\|_{\dot H^s_y}^2
\end{equation}
for a suitable $c>0$.
We apply the Plancharel identity and we get
$$\int_{0}^{2\pi}|u(x+h)-u(x)|^2 dx=
\sum_n |e^{inh}-1|^2 |\hat u(n)|^2
$$
and hence
$$\int_0^{2\pi} \int_\R \frac{|u(x+h)-u(x)|^2}{|h|^{1+2s}} dxdh
= \sum_n |\hat u(n)|^2 \int_{\R}|e^{inh}-1|^2  \frac{dh}{|h|^{1+2s}}.$$
Next notice that
\begin{align*}\int_{\R}|e^{inh}-1|^2 \frac{dh}{|h|^{1+2s}} 
&= \int_{\R}|e^{inh}-1|^2 \frac{|n|^{1+2s} |n|dh}{|n||nh|^{1+2s}} \\\nonumber 
&= |n|^{2s} \int_{\R} |e^{ir}-1|^2 \frac{dr}{r^{1+2s}}= c|n|^{2s}
\end{align*}
and hence by combining the identities above
we get \eqref{equivHs.}.
Based on \eqref{equivHs.} we get:
\begin{align*}\|u|u|^\alpha\|_{\dot H^s_y}^2&=c
\int_0^{2\pi} \int_\R \frac{|u|u|^\alpha(x+h)-u|u|^\alpha(x)|^2}{|h|^{1+2s}} dxdh
\\\nonumber &\leq C \int_0^{2\pi} \int_\R \frac{|u(x+h)-u(x)|^2 \|u\|_{L^\infty}^{2\alpha}}{|h|^{1+2s}} dxdh
\leq C \|u\|_{\dot H^s_y}^2 \|u\|_{L^\infty}^{2 \alpha}.
\end{align*}
\end{proof}

\noindent{\em Proof of Theorem \ref{halflwpStrong}.}
\\
\\
{\bf First case: $4/d\leq \alpha<4/(d-1)$}
\\
\\
In the sequel we shall denote by $X^{1/2+\delta}_T(q,r)$ the space whose norm is defined as follows:
\begin{equation} 
\| u\|_{X^{1/2+\delta}_T(q,r)} = \|u(t,x, y)\|_{ L^q_T L^{r}_xH_y^{1/2+\delta}},
\end{equation}
with $\delta>0,$ $T>0$. Here we use the notation $L^q_T(X)=L^q((-T,T);X)$.\\
From now on $(q,r)$ will be any couple given
by Proposition \ref{AlgebricCauchyfer} and $\delta>0$ will be in such a way that
$1/2+\delta+s\leq 1$, where $0<s<1/2$
is defined by Proposition~\ref{AlgebricCauchyfer}.\\
We shall also need the following localized norms $Y^{(1)}_T(\ell,p)$
and $Y^{(2)}_T(\ell,p)$:
\begin{align*}
\| u\|_{Y^{(1)}_T(\ell,p)} = \sum_{k=0,1}\sum_{j=1}^d\|\partial_{x_j}^k u(t,x, y)\|_{{ L^\ell_t( (-T, T), L^{p}_xL_y^2)}},& \\\nonumber
\| u\|_{Y^{(2)}_T(\ell,p)} =  \sum_{k=0,1}\|\partial_y^ku(t,x, y)\|_{{ L^\ell_t( (-T, T), L^{p}_xL_y^2)}},&
\end{align*}
where $(\ell, p)$  are associated with $(q,r)$ via  Lemma \ref{sTrbis}.\\
We also set the global norm:
\begin{align*}
\label{eq.1st2}
\|w\|_{Z_T^{1/2+\delta}}=\|w\|_{X^{1/2+\delta}_T(q,r)}+\|w\|_{Y^{(1)}_T(\ell,p)}+\|w\|_{Y^{(2)}_T(\ell,p)},
\end{align*}
and we introduce the integral operator:
\begin{equation}\label{integralFrO}
\mathcal{A}_fu=e^{-it\Delta_{x,y}}f+ i \int_{0}^{t} e^ {-i
(t-\tau) \Delta_{x,y}} \big(u(\tau)|u(\tau)|^\alpha\big) d\tau.
\end{equation}
We split the proof in four steps.
\\
\\
\emph {Step One: $\forall f \in H^{1}_{x,y}$ \hbox{ }  $\exists T=T\big( \|f\|_{H^{1}_{x,y}}\big)>0$ and $R=R\big( \|f\|_{H^{1}_{x,y}}\big)>0$ s.t. 
$
\mathcal{A}_f(B_{Z^{1/2+\delta}_{T'}})\subset B_{Z^{1/2+\delta}_{T'}},
$
for any $T'<T.$}
\\
\\
Let $\tilde q, \tilde r$ be the ones given by Propositon \ref{AlgebricCauchyfer}. We start by noticing that
 \begin{align}\label{eq.1nonl}
\|u|u|^\alpha\|_{L^{\tilde{q}'}_T
L^{\tilde{r}'}_xH_y^{1/2+\delta}}  \leq C
  \| \|u(t,x,.)\|_{H_y^{1/2+\delta}}^{\alpha+1}\|_{L^{\tilde{q}'}_TL^{\tilde{r}'}_x},
\end{align}
where  we used Lemma \ref{TzBil}.
By combining this estimate with  \eqref{eq.ass9bisvfer} and
with the H\"older inequality we get:
\begin{align}\label{eq.2nonla}
\|u|u|^\alpha\|_{L^{\tilde{q}'}_T
L^{\tilde{r}'}_xH_y^{1/2+\delta}} \leq C
  \| \|u\|_{L^{r}_xH_y^{1/2+\delta}}^{\alpha+1}\|_{L^{\tilde{q}'}_T}
  \leq CT^{\beta(\alpha)}\|u\|_{L^{q}_TL^{r}_xH_y^{1/2+\delta}}^{\alpha+1},
\end{align}
with $\beta(\alpha)>0$ and for some constant $C>0$ independent on $T.$ 
By combining this nonlinear estimate with Propositions \ref{StriProd}
and \ref{StriProd2},
we conclude the following:
\begin{align}\label{eq.2nonl}
  \|\mathcal{A}_fu\|_{X^{1/2+\delta}_T(q,r)} \ \leq
    C \|f\|_{H_x^{s} H_y^{1/2+\delta}} + C  T^{\beta(\alpha)} \|u\|_{X^{1/2+\delta}_T(q,r)}^{\alpha+1}.
\end{align}
A combination of Proposition \ref{StriProd1} with Lemma \ref{sTrbis}, in conjunction with 
the H\"older inequality,
yield the following estimate:
\begin{align}
\label{eq.nonl3}\|\mathcal{A}_fu\|_{Y^{(i)}_T(\ell,p)} \leq C \sum_{k=0,1} (
\|D^k f\|_{L^2_{x,y}} + \|D^k (u|u|^\alpha)\|_{L^{\ell'}_TL^{p'}_xL^2_y})&\\
\leq C \sum_{k=0,1} (\|D^kf\|_{L_{x,y}^2} + \big \|\|D^k u(t,x, y)\|_{L^2_y} \|u(t, x, y)\|_{L^\infty_y}^\alpha \big \|_{L^{\ell'}_TL^{p'}_x})&\nonumber\\
\leq C \sum_{k=0,1} (\|D^kf\|_{L_{x,y}^2} + \big \|\|D^ku(t,x, y)\|_{L^2_y} \|u(t, x, y)\|_{H^{1/2+\delta}_y}^\alpha \big \|_{L^{\ell'}_TL^{p'}_x})&\nonumber\\
\leq C (\sum_{k=0,1} \|D^kf\|_{L_{x,y}^2} +  T^{\beta(\alpha)} \|D^k u(t,x, y)\|_{L^{\ell}_TL^{p}_xL^2_y} \|u(t, x, y)\|_{L^{q}_TL^{r}_x H^{1/2+\delta}_y}^\alpha),&\nonumber
\end{align}
where $D$ stands for $\partial_y, \partial_{x_j}$, $j=1,..,d$,  $k=0,1$ and in the third inequality we used the embedding $H^{1/2+\delta}_y\subset L^\infty_y.$
Hence we get
\begin{align}\label{eq.2nonl2}
  \|\mathcal{A}_fu\|_{Y^{(i)}_T(\ell,p)} \leq
  C \|f\|_{H_{x,y}^1} + C T^{\beta(\alpha)}\|u\|_{Y^{(i)}_T(\ell,p)}\|u\|_{X^{1/2+\delta}_T(q,r)}^{\alpha}.
 \end{align}
We can conclude the proof of this step 
by combining 
\eqref{eq.2nonl} and \eqref{eq.2nonl2}. 
 \\
\\
\emph {Step Two: let $T, R >0$ as in the previous step then there exist
 $\overline T = \overline T \big (\|f\|_{H^{1}_{x,y}}\big)<T$
such that $\mathcal{A}_f$ is a contraction on $B_{Z^{1/2+\delta}_{\overline T}}(0,R),$ equipped
 with the norm  $\| . \|_{ L^q_{\overline T} L^{r}_xL_y^{2}}$.}
 \\
 \\
 Given any $v_1,v_2\in B_{X^{1/2+\delta}_T(q,r)}(0,R)$  we achieve, by an use of estimate \eqref{eq:202p}, the chain of bounds:
\begin{align*}
  \|\mathcal{A}_fv_1-\mathcal{A}_f v_2\|_{L^q_T L^r_xL_y^{2}} \leq C\|v_1|v_1|^\alpha-v_2|v_2|^\alpha\|_{L^{\tilde{q}'}_T L^{\tilde{r}'}_x L_y^{2}} & \nonumber\\
\leq C \big \|\|v_1-v_2\|_{ L_y^{2}}(\|v_1\|_{L^\infty_y}^\alpha
+\|v_2\|_{L^\infty_y}^\alpha)\big \|_{L^{\tilde{q}'}_TL^{\tilde{r}'}_x}\leq&\nonumber\\
\leq C \big \|\|v_1-v_2\|_{L^{r}_x L_y^{2}}(\|v_1\|_{L^{r}_x H_y^{1/2+\delta}}^\alpha+
\|v_2\|_{L^{r}_x H_y^{1/2+\delta}}^\alpha)\big \|_{L^{\tilde{q}'}_T}&
\end{align*}
where we used 
the Sobolev embedding $H^{1/2+\delta}_y\subset L^\infty_y$
and \eqref{eq.ass9bisvfer} at the last step.
Again by the H\"older inequality in conjunction with \eqref{eq.ass9bisvfer}, we can continue the estimate
as follows: 
\begin{align}\label{eq.5nonl}
 ...\leq C T^{\beta(\alpha)}\left (\|v_1\|_{L^{q}_TL^{r}_x H_y^{1/2+\delta}}^\alpha+\|v_2\|_{L^{q}_TL^{r}_x H_y^{1/2+\delta}}^\alpha\right)\|v_1-v_2\|_{L^{q}_TL^r_xL_y^{2}}&
\end{align}
and we can conclude.
\\
\\
\emph {Step Three: the solution exists and is unique in $Z_{\overline{T}}^{1/2+\delta},$ where $\overline{T}$ is as in the above step}. 
\\
\\
We are in position to show existence and uniqueness of the solution by applying the contraction principle to the map $\mathcal A_f$ defined on the complete metric space $B_{Z_{\overline T}^{1/2+\delta}}(0,R)$, equipped with the topology induced by 
$\| . \|_{ L^q_{\overline T} L^{r}_xL_y^{2}}.$
\\
\\
\emph {Step Four: $u(t,x,y)\in {\mathcal C}( (-T, T);H^{1}_{x,y}).$}
\\
\\
Arguing
as in the proof of \eqref{eq.nonl3} 
we get:
\begin{align}\label{eq.2nonl3}
 \|\mathcal{A}_fu\|_{L^\infty((-T, T), L^2_{x,y})} &+ 
 \sum_{j=1}^d
 \|\partial_{x_j} \mathcal{A}_fu\|_{L^\infty((-T, T), L^2_{x,y})}\\\nonumber +\|\partial_y\mathcal{A}_fu\|_{L^\infty((-T, T), L^2_{x,y})}\leq
 &C \|f\|_{H_{x,y}^1} + C T^{\beta(\alpha)} \|u\|_{Z_T^{1/2+\delta}}\|u\|_{X^{1/2+\delta}_T(q,r)}^{\alpha}.
  \end{align} 
This estimate it is sufficient to guarantee that $u(t,x,y)\in {\mathcal C}((-T, T);H^{1}_{x,y}).$\\
\\
The last step is the proof of unconditional uniqueness of solutions
to \eqref{NLS}.
\\
\\
\emph {Step Four: if $u_1, u_2\in {\mathcal C}( (-T, T);H^{1}_{x,y})$ are fixed points of
${\mathcal A}_f$ then $u_1=u_2$.}
\\
\\
By a continuity argument it is sufficient to show
that $u_1(t)=u_2(t)$ for a short time $(-\tilde T, \tilde T)$, where
$\tilde T$ depends only on the $H^1_{x,y}$ norms of $f$.\\
By taking the difference of the integral equations satisfied by $u_1$ and $u_2$ we get
\begin{align}\label{difference}
(u_1-u_2)(t,x,y)=\int_0^t e^{-i(t-\tau)\Delta_{x,y}}({u_1(\tau)|u_1(\tau)|^\alpha- u_2(\tau)|u_2(\tau)|^\alpha})d\tau.
\end{align}
By an application of Proposition \ref{StriProd1} we get
\begin{align}\label{eq.un1}
\|u_1-u_2\|_{L_T^{l}L_{x}^{p} L_y^{2}}
\leq C \||u_1|^\alpha u_1-|u_2|^\alpha u_2\|_{L_T^{l'}L_{x}^{p'} L_y^{2}}
\end{align}
provided that $\frac 2l+\frac dp=\frac d2$, $l\geq 2$, $(l, d)\neq (2, 2)$.
We can continue \eqref{eq.un1} as follows
\begin{align}\label {eq.un3}
...\leq 
C\| u_1-u_2\|_{L_T^{l}L_{x}^{p} L_y^{2}}\big (\sum_{j=1,2} 
\|u_j\|_{ L_T^{\frac {\alpha l}{l-2}}L_{x}^{\frac{\alpha p}{p-2}} L_y^{\infty}}^\alpha\big)
\\\nonumber \leq C\| u_1-u_2\|_{L_T^{l}L_{x}^{p} L_y^{2}}
T^{\frac{l-2}{l}}\big (\sum_{j=1,2} 
\|u_j\|_{ L_T^{\infty}L_{x}^{\frac{\alpha p}{p-2}} H_y^{1/2+\delta}}^\alpha\big),
\end{align}
where we used $H^{1/2+\delta}_y\subset L^\infty_y$.
We conclude the proof of uniqueness by selecting $T$ small enough and 
$\delta, p$ in such a way that
$$\|v\|_{L_{x}^{\frac{\alpha p}{p-2}} H_y^{1/2+\delta}}\leq
C \|v\|_{H^1_{x,y}}.$$ 
Indeed the estimate above follows by combining \eqref{lysi} and 
the trivial estimate $$\|v\|_{L^2_xH^{1/2+\gamma}_y}\leq 
\|v\|_{H^1_{x,y}} \forall \gamma>0$$
provided that we can select $p$ in such a way that
\begin{equation}\label{kapp}
2<\frac{\alpha p}{p-2}<\frac{2d}{d-1}.\end{equation}
Notice that the values allowed to $p$ are the following:
\begin{align*}
p\in [2, \infty] \hbox{ for } d=1\\\nonumber
p\in [2, \infty) \hbox{ for } d=2\\\nonumber
p\in [2, 2d/(d-2)] \hbox{ for } d\geq 3.
\end{align*}
Hence for $d=1$ we can trivially satisfy \eqref{kapp} for a suitable $p$.
For $d=2$ notice that $\lim_{p\rightarrow 2} \frac{\alpha p}{p-2}=\infty$
and $\lim_{p\rightarrow \infty} \frac{\alpha p}{p-2}=\alpha$ and we 
can guarantee \eqref{kapp} for a suitable $p$ since
$0<\alpha<4/(d-1)=2d/(d-1)$ for $d=2$.
In the case $d\geq 3$ we get $\lim_{p\rightarrow 2} \frac{\alpha p}{p-2}=\infty$
and $\lim_{p\rightarrow 2d/(d-2)} \frac{\alpha p}{p-2}=\alpha d/2$. We conclude
since $\alpha d/2<2d/(d-1)$ (indeed it is equivalent to the assumption $\alpha<4/(d-1)$).
\\
\\
{\bf Second case: $0<\alpha<4/d$.}
\\
\\
The proof is similar to the case $4/d\leq \alpha<4/(d-1)$ with minor changes. 
In this case the space
$X_T^{1/2+\delta}(q,r)$ is selected with a couple $(q,r)$ given by Proposition \ref{strichclass}.
Indeed we use
Proposition \ref{strichclass} instead of Proposition \ref{AlgebricCauchyfer},
and we use on the Duhamel operator the estimates 
in Proposition \ref{StriProd1} instead of the ones in Proposition \ref{StriProd2}.
On the linear propagator we use Proposition \ref{StriProd1}
instead of Proposition \ref{StriProd}. The proof of the unconditional uniqueness
provided in the previous step works for every $0<\alpha<4/(d-1)$.

\section{Interaction Morawetz Estimates and proof of Proposition \ref{decay}}
Along this section we shall denote by $\int$ the integral with respect to $dxdy$
and by $\int \int$ the integral with respect to $dx_1 dy_1 dx_2 dy_2$.
For $x\in\R^d$ and $r\geq 0$, we define $Q^d(x,r)$ to be a $r$ dilation of the unit cube centered at $x$, namely 
$$Q^d(x,r)=x+[0,r]^d\,.$$
The next lemma contains the key global information needed for our analysis.
\begin{lem}
Let $u(t,x,y)\in {\mathcal C}(\R;H^1_{x,y})$ be as in Proposition \ref{decay}.
Then for any $\psi\in C^\infty_0(\R^d)$ we get: 
\begin{equation}\label{1}
\frac d{dt} \int \psi(x) |u(t,x,y)|^2 dxdy=-2 
{\mathcal Im}\int \bar u \nabla_x \psi \cdot \nabla_x u dx dy.\end{equation}
Moreover we have:
\begin{equation}\label{2} -2 
\frac d{dt} {\mathcal Im}\int \bar u \nabla_x (\langle x \rangle) \cdot \nabla_x u dx dy
\end{equation}$$=
4  \int \nabla_x u D^2_x  (\langle x \rangle) \nabla_x \bar u dxdy -
\int \Delta_x^2  (\langle x \rangle) |u|^2 dxdy
+ \frac{2\alpha}{\alpha+2}
\int \Delta_x  (\langle x \rangle) |u|^{\alpha+2} dxdy.
$$
Moreover, for every $r>0$, there exists $C$ such that
\begin{align}\label{inter1}
\int_{\R} &\big ( \sup_{x_0\in \R^d} \int\int_{Q^d(x_0,r)\times (0, 2\pi)}
|u(t, x, y)|^{2}dxdy\big )^\frac{\alpha+4}2 dt
\leq C \|f\|_{H^1_{x,y}}^4\,.
\end{align}
\end{lem} 
\begin{remark}\label{rigourint}
By \eqref{1} we see that the l.h.s. in \eqref{2} can be considered, at least formally, as
the second derivative w.r.t. time of 
$\int \int \langle x \rangle |u(t, x,y)|^2 dxdy$, which is not a well defined quantity
for $u\in H^1_{x,y}$.
However, the quantity involved on the l.h.s. in \eqref{2} is well-defined since
$\nabla_x \langle x\rangle\in L^\infty_x$ and $u\in H^1_{x,y}$. 
For this reason we have decided to write
in terms of first derivative \eqref{1} and \eqref{2}.
\end{remark}
\begin{proof}
The proof of \eqref{1} and \eqref{2} follow by standard considerations and we skip it.
Concerning the proof of \eqref{inter1} we follow \cite{Iteam}. From now on
we define $\varphi(x)=\langle x \rangle$
and we make some formal computations. At the end of the proof
we shall explain how to make rigorous the arguments below. Write 
\begin{align}\nonumber \frac{d}{dt} \int \int |u(t,x_1, y_1)|^2 \varphi(x_1-x_2) 
|u(t, x_2, y_2)|^2 dx_1dx_2dy_1dy_2\\\nonumber 
= \int \big (\int \frac d{dt} |u(t, x_1, y_1)|^2 \varphi(x_1-x_2) dx_1dy_1\big )
|u(t, x_2, y_2)|^2 dx_2dy_2 \\\nonumber 
+ \int \big (\int \frac d{dt} |u(t, x_2, y_2)|^2 \varphi(x_1-x_2)  dx_2dy_2\big )
|u(t, x_1, y_1)|^2 dx_1dy_1.\nonumber \end{align}
From now on we shall drop the variable $t$ for simplicity and hence
we shall write $u(t, x_i, y_i)=u(x_i, y_i)$.
By combining the identity above with \eqref{1} we get
\begin{align}\label{ddt}\frac{d}{dt} \int \int |u(x_1, y_1)|^2 \varphi(x_1-x_2) 
|u(x_2, y_2)|^2 dx_1dx_2dy_1dy_2\\\nonumber 
= -2 
{\mathcal Im}\int \int \bar u(x_1, y_1)\nabla_{x_1} u(x_1, y_1) \cdot \nabla_{x} 
\varphi(x_1-x_2)
|u(x_2, y_2)|^2  dx_1 dy_1 dx_2dy_2 \\\nonumber +2 
{\mathcal Im}\int \int \bar u(x_2, y_2)\nabla_{x_2} u(x_2, y_2) \cdot \nabla_{x} 
\varphi(x_1-x_2) 
|u(x_1, y_1)|^2 dx_1dy_1 dx_2 dy_2.\end{align}
Next notice that
\begin{align}\label{iMpO}\frac{d^2}{dt^2} \int \int |u(x_1, y_1)|^2 \varphi(x_1-x_2) 
|u(x_2, y_2)|^2 dx_1dx_2dy_1dy_2\\\nonumber 
=  \int \big (\int  \frac{d^2}{dt^2} |u(x_1, y_1)|^2 \varphi(x_1-x_2) dx_1 dy_1\big )
|u(x_2, y_2)|^2 dx_2 dy_2  \\\nonumber
+\int \big (\int \frac{d^2}{dt^2} |u(x_2, y_2)|^2 \varphi(x_1-x_2) dx_2 dy_2\big )
|u(x_1, y_1)|^2 dx_1 dy_1 \\
\nonumber  
+ 2  \int \big (\frac{d}{dt} \int  |u(x_1, y_1)|^2 \varphi(x_1-x_2) dx_1 dy_1\big )
\frac d{dt}|u(x_2, y_2)|^2 dx_2 dy_2\\\nonumber=I+II+III.\end{align}
 By using \eqref{1} twice we get
\begin{align}\label{carl1}III
\\\nonumber =-4 \int \frac d{dt} |u(x_2, y_2)|^2 
\big ({\mathcal Im}\int \bar u(x_1, y_1) \nabla_{x_1} \varphi(x_1-x_2) \cdot  \nabla_{x_1} 
u(x_1, y_1) dx_1 dy_1\big ) dx_2 dy_2\\\nonumber 
= 8 
{\mathcal Im}\int \bar u(x_2, y_2) \nabla_{x_2} u(x_2, y_2)  \cdot
\nabla_{x_2} ({\mathcal Im}\int F(x_1, x_2, y_1)dx_1 dy_1) dx_2 dy_2
\\\nonumber
=-8 \int \int  
V(x_1, y_1)D^2_x \varphi (x_1-x_2) V(x_2, y_2)
dx_1 dy_1dx_2dy_2,
\end{align}
where \begin{align}
\nonumber F(x_1, x_2, y_1)=\bar u(x_1, y_1) \nabla_{x_1} \varphi(x_1-x_2)
\cdot  \nabla_{x_1} 
u(x_1, y_1)\\\nonumber
V(x,y)={\mathcal Im} (\bar u(x,y) \nabla_{x} u(x,y)).
\end{align}
Moreover the term $I$ in the r.h.s. of \eqref{iMpO} can be rewritten as follows:
\begin{align}\nonumber
I\\\nonumber =4  \int \int \nabla_{x_1} u(x_1, y_1) D^2_x \varphi (x_1-x_2)
 \nabla_{x_1} \bar u(x_1, y_1)|u(x_2, y_2)|^2 dx_1dy_1 dx_2dy_2 \\\nonumber-
 \int \int  \Delta_x^2 \varphi(x_1-x_2) |u(x_1, y_1)|^2 |u(x_2, y_2)|^2 dx_1dy_1 dx_2dy_2
\\\nonumber + \frac{2\alpha}{\alpha+2} \int \int
\Delta_x \varphi (x_1-x_2)
|u(x_1, y_1)|^{\alpha+2} |u(x_2, y_2)|^2 dx_1dy_1 dx_2dy_2,
\end{align}
that by the following identity (obtained by integration by parts, see \cite{PV})
$$-\int \int \Delta^2_x \varphi(x_1-x_2) |u(x_1, y_1)|^2|u(x_2, y_2)|^2 dx_1dy_1dx_2dy_2$$
$$= \int\int \nabla_{x_1} (|u(x_1, y_1)|^2) D^2_x \varphi(x_1-x_2) \nabla_{x_2} (|u(x_2, y_2)|^2)
dx_1 dy_1dx_2 dy_2$$
becomes:
\begin{align}\label{carl2}I\\\nonumber =4  \int \int  \nabla_{x_1} u(x_1, y_1) D^2_x \varphi(x_1-x_2) \nabla_{x_1} \bar u(x_1, y_1)|u(x_2, y_2)|^2 dx_1 dy_1dx_2 dy_2 \\\nonumber
 +\int \int \nabla_{x_1} (|u(x_1, y_1)|^2) D^2_x \varphi(x_1-x_2)\nabla_{x_2} (|u(x_2, y_2)|^2)
dx_1 dy_1dx_2 dy_2\\\nonumber
+ \int \int \frac{2\alpha}{\alpha+2}
\Delta_x \varphi (x_1-x_2)|u(x_1, y_1)|^{\alpha+2} |u(x_2, y_2)|^2 dx_1dy_1 dx_2dy_2.
\end{align}
By similar arguments the term $II$ in the r.h.s. of \eqref{iMpO} can be rewritten as follows:
\begin{align}\label{carl3}II\\\nonumber
=4  \int \int \nabla_{x_2} u(x_2, y_2) D^2_x \varphi(x_1-x_2) \nabla_{x_2} \bar u(x_2, y_2)|u(x_1, y_1)|^2 \\\nonumber+\int \int \nabla_{x_1} |u(x_1, y_1)|^2 D^2_x \varphi(x_1-x_2)\nabla_{x_2} |u(x_2, y_2)|^2
\\\nonumber+ \frac{2\alpha}{\alpha+2}
\int \int \Delta_x\varphi(x_1-x_2) |u(x_2, y_2)|^{\alpha+2} |u(x_1, y_1)|^2 dx_1dy_1 dx_2dy_2.
\end{align}
Next we introduce the vectors
$A(t,x_1, y_1, x_2, y_2), B(t,x_1, y_1, x_2, y_2)$ defined as follows:
$$A(t,x_1, y_1, x_2, y_2):=u(x_1, y_1) \nabla_{x_2}\bar u(t, x_2, y_2)
+ \bar u(x_2, y_2)\nabla_{x_1} u(x_1, y_1)
$$
and 
$$B(t,x_1, y_1, x_2, y_2):=u(x_1, y_1) \nabla_{x_2} u(x_2, y_2)
- u(x_2, y_2)\nabla_{x_1} u(x_1, y_1).$$
By direct computation we get
\begin{align}\label{guam}
2 A D^2_x \varphi(x_1-x_2) \bar A
+2 B D^2_x \varphi(x_1-x_2) \bar B\\
\nonumber = 4 \nabla_{x_1} u(x_1, y_1)
D^2_x \varphi(x_1-x_2)  \nabla_{x_1} \bar u(x_1, y_1)|u(x_2, y_2)|^2\\
\nonumber + 4 
\nabla_{x_2} u(x_2, y_2) D^2_x \varphi(x_1-x_2) \nabla_{x_2} \bar u(x_2, y_2)|u(x_1, y_1)|^2
\\
\nonumber 
-8 \big ({\mathcal Im} \bar u(x_1, y_1) \nabla_{x_1} u(x_1, y_1)\big )
D^2_x \varphi(x_1-x_2)
\big ({\mathcal Im} \bar u(x_2,y_2) \nabla_{x_2} u(x_2,y_2)\big)
\end{align}
and also
\begin{equation}\label{giam}2 A D^2_x \varphi(x_1-x_2) \bar A
+2 B D^2_x \varphi(x_1-x_2) \bar B\end{equation}$$+2 \nabla_{x_1} |u(x_1, y_1)|^2
D^2_x \varphi(x_1-x_2) \nabla_{x_2} |u(x_2, y_2)|^2)
=4 A D^2_x \varphi(x_1-x_2) \bar A\geq 0.$$
By combining \eqref{guam} and \eqref{giam}  we get
\begin{align}\label{prca}4  \nabla_{x_1} u(x_1, y_1)  
D^2_x \varphi(x_1-x_2) \nabla_{x_1} \bar u(x_1, y_1)|u(x_2, y_2)|^2 \\\nonumber
+ 4  
\nabla_{x_2} u(x_2, y_2) D^2_x \varphi(x_1-x_2) \nabla_{x_2} \bar u(x_2, y_2)|u(x_1, y_1)|^2\\\nonumber
-8 ({\mathcal Im} \bar u(x_1, y_1) \nabla_{x_1} u(x_1, y_1))
D^2_x \varphi(x_1-x_2)
({\mathcal Im} \bar u(x_2,y_2) \nabla_{x_2} u(x_2,y_2)\\\nonumber
+2 \nabla_{x_1} (|u(x_1, y_1)|^2) D^2_x \varphi(x_1-x_2) 
\nabla_{x_2} (|u(x_2, y_2)|^2)
\geq 0,\end{align} and hence
by \eqref{carl1}, \eqref{carl2}, \eqref{carl3} and \eqref{prca} we obtain
$$\frac{d^2}{dt^2} \int\int |u(x_1, y_1)|^2 \varphi(x_1-x_2) 
|u(x_2, y_2)|^2 dx_1dx_2dy_1dy_2=I+II+III $$
$$\geq \frac{4\alpha}{\alpha+2}
\int \int \Delta_x \varphi (x_1-x_2)|u(x_1, y_1)|^{\alpha+2} |u(x_2, y_2)|^2 dx_1dy_1 dx_2dy_2.$$
Integration in time  gives
\begin{align}\label{impbe}\frac{d}{dt} (\int \int |u(x_1, y_1)|^2 \varphi(x_1-x_2) 
|u(x_2, y_2)|^2 dx_1dx_2dy_1dy_2){_{t=\infty}}
\\\nonumber - \frac{d}{dt} (\int \int |u(x_1, y_1)|^2 \varphi(x_1-x_2) 
|u(x_2, y_2)|^2 dx_1dx_2dy_1dy_2){_{t=0}}\\\nonumber =\int (I+II+III) dt\\
\nonumber \geq \frac{4\alpha}{\alpha+2}
\int \int \int \Delta_x \varphi (x_1-x_2)|u(x_1, y_1)|^{\alpha+2} |u(x_2, y_2)|^2 dt dx_1dy_1 dx_2dy_2.\end{align}
Notice that by \eqref{ddt} the l.h.s. can be controlled by $C\|f\|_{H^1_{x,y}}^4$
provided that we choose $\varphi=\langle x \rangle$.
On the other hand 
we have $\inf_{Q^d(0,2r)} \Delta_x (\langle x \rangle)>0$, hence we get 
\begin{align}
\int_{\R} \sup_{x_0\in \R^d}&\big (\int\int_{(Q^d(x_0,r))^2\times (0, 2\pi)^2}
|u(x_2, y_2)|^{\alpha+2} |u(x_1, y_1)|^{2}dx_1dx_2 dy_1dy_2\big ) dt
\nonumber
\\\nonumber
&\leq C \|f\|_{H^1_{x,y}}^4
\end{align}
where we used the notation $A^2=A\times A$ for any general set $A$.
In turn by the H\"older inequality we get
$$\int_{Q^d(x_0,r)\times (0, 2\pi)}  |u(x_2, y_2)|^{\alpha+2}dx_2 dy_2 \geq C_{r}
(\int_{Q^d(x_0,r)\times (0, 2\pi)} |u(x_2, y_2)|^{2}dx_2 dy_2)^\frac{\alpha+2}{2}$$
and we conclude the proof of \eqref{inter1}.

Indeed, the computation above is formal since the quantity 
$$\int \int |u(x_1, y_1)|^2 \langle x_1-x_2 \rangle 
|u(x_2, y_2)|^2 dx_1dx_2dy_1dy_2$$
appearing in \eqref{ddt}
it is not well-defined
 for $u\in H^1_{x,y}$.
Following the Remark~\ref{rigourint}, 
we can make rigorous the argument above
by writing the following identity
$$\frac d{dt} J(t)=I+II+III,$$
where the quantity
\begin{align*}J(t)=-2 
{\mathcal Im}\int \int \bar u(x_1, y_1)\nabla_{x_1} u(x_1, y_1) \cdot \nabla_{x} 
\varphi(x_1-x_2)
|u(x_2, y_2)|^2  dx_1 dy_1 dx_2dy_2 \\\nonumber +
2{\mathcal Im}\int \int \bar u(x_2, y_2)\nabla_{x_2} u(x_2, y_2) \cdot \nabla_{x} 
\varphi(x_1-x_2) 
|u(x_1, y_1)|^2 dx_1dy_1 dx_2 dy_2\end{align*}
is meaningful for $u\in H^1_{x,y}$.

\end{proof}

\noindent{\em Proof of Proposition \ref{decay}.\,}
We follow the approach in \cite{Vis}.
First, we write the following well-known localized Gagliardo-Nirenberg inequality:
\begin{equation}\label{GNprecised}
\|v\|_{L^{2+4/(d+1)}_{x,y}}\leq C \sup_{x\in \R^{d}} 
\Big (\|v\|_{L^2_{Q^{d}(x,1)\times (0,2\pi)}}\Big )^{2/(d+3)} 
\|v\|_{H^1_{x,y}}^{(d+1)/(d+3)}.
\end{equation}
Of course it is sufficient to show that
\begin{equation}\label{part}
\lim_{t\rightarrow \pm \infty} \|u(t, x,y)\|_{L^{2+4/(d+1)}_{x,y}}=0.
\end{equation}
In fact the decay of the $L^q_{x,y}$ norm for $2<q<\frac{2(d+1)}{d-1}$ follows
by combining \eqref{part} with the bound 
\begin{equation}\label{unifbound}\sup_{t\in \R} \|u(t, x,y)\|_{H^1_{x,y}}
<\infty.\end{equation}
Next, assume by the absurd that \eqref{part} is false, then
by \eqref{GNprecised} and by \eqref{unifbound} we deduce 
the existence of a sequence $(t_n, x_n)\in \R\times \R^d$ with
$|t_n|\rightarrow \infty$  and $\epsilon_0>0$ such that
\begin{equation}\label{loc}
\inf_n \|u(t_n,x,y)\|_{L^2_{Q^{d}(x_n,1)\times (0, 2\pi)}}
=\epsilon_0.
\end{equation}
For simplicity we can assume that $t_n\rightarrow \infty$ (the case 
$t_n\rightarrow -\infty$ can be treated by a similar argument).\\
Notice that by \eqref{1} in conjunction with \eqref{unifbound}
we get
$$\sup_{n,t} \big | \frac{d}{dt} \int \chi(x-x_n) |u(t, x, y)|^2dx dy
\big |<\infty,$$
where $\chi(x)$ is a smooth and non-negative cut-off function taking values in $[0,1]$ such that
$\chi(x)=1$ for $x\in Q^d(0,1)$ and $\chi(x)=0$ for $x\notin Q^d(0,2)$.
By combining this fact with \eqref{loc} then we get the existence of $T>0$ such that
\begin{equation}\label{locloc}\inf_{n} \big (\inf_{t\in (t_n,t_n+T)}  \|u(t,x,y)\|_{L^2_{Q^{d}(x_n,2)\times (0, 2\pi)}}\big )\geq \epsilon_0/2.
\end{equation}
Notice that since $t_n\rightarrow \infty$ then we can assume (modulo subsequence)
that the intervals 
$(t_n, t_n+T)$ are disjoint and hence we get
a contradiction with \eqref{inter1}.
\hfill$\Box$
\section{Fixing the admissible exponents for the scattering analysis}
In this section we prepare some result useful to prove Theorem \ref{main}.
\begin{prop}\label{Algebric0}
Let $d\geq 1$ and $4/d<\alpha<4/(d-1)$ be fixed, and
$s=\frac{\alpha d-4}{2\alpha}$.
Then there exists $\theta\in (0,1)$ 
and $(q_\theta, r_\theta, \tilde{q}_\theta, \tilde{r}_\theta),$ in such a way that: \begin{equation}\label{easyconitionp} 0<\frac 1{q_\theta},\frac 1
{r_\theta}, \frac 1{\tilde q_\theta}, \frac 1{\tilde r_\theta}<\frac 12
\end{equation} 
\begin{align}
\label{fewr}\frac 1 {q_\theta}+ \frac 1{\tilde q_\theta}<1, \quad \quad 
& \frac{d-2}{d}<\frac{r_\theta}{\tilde{r}_\theta}< \frac{d}{d-2}\\
\label{dropd12ferro}
\frac{1}{q_\theta}+\frac{d}{r_\theta}<\frac{d}{2}, \quad \quad & 
\frac{1}{\tilde{q}_\theta}+\frac{d}{\tilde{r}_\theta}<\frac{d}{2}\\
\label{eq.ass102v3ferro}  \frac{2}{q_\theta} +\frac{d}{r_\theta}=\frac d2- s,\quad \quad &
\frac{2}{q_\theta} + \frac d{r_\theta} + \frac{2}{\tilde{q}_\theta} +\frac d{\tilde r_\theta}=d,\\
\label{eq.ass9bispferro} \frac{1}{ (\alpha+1)\tilde{q}_\theta'}= \frac{\theta}{q_\theta}, \quad \quad& 
\frac1{(\alpha+1)\tilde r_\theta'}=\frac{\theta}{r_\theta} +\frac{2(1-\theta)}{\alpha d}.&
\end{align}
For $d=1,2$ we get the same conclusion provided that we drop 
conditions \eqref{fewr}.\\
Moreover we can also assume that
\begin{equation}\label{impBosUl3}
\frac{\alpha}{q_\theta}+\frac{\alpha d}{2r_\theta}=1,\quad \quad \ \frac \alpha {r_\theta}<1.
\end{equation}
\end{prop}

\begin{remark}\label{appl}
By combining Propositions \ref{StriProd} and
\ref{StriProd2}, we get the following estimate:
\begin{align}\label{eq:202p90}
 &\|e^{-it\Delta_{x,y}} f\|_{L^{q_\theta}_t L^{r_\theta}_xH_y^{\gamma}} +\left \|\int_0^t e^{ -i (t-\tau) \Delta_{x,y}} F(\tau)
d\tau\right\|_{L^{q_\theta}_t L^{r_\theta}_xH_y^{\gamma}}\\
\nonumber &\leq C(\|f\|_{ H_x^{s}H_y^{\gamma}}+\|F\|_{L^{\tilde{q_\theta}'}_t
L^{\tilde{r}_\theta'}_xH_y^{\gamma}})
\end{align}
for every $\gamma\in \R$.
\end{remark}

\noindent{\em Proof of Proposition \ref{Algebric0}.}
For the moment we let $\theta$ to be free, and at the end we shall select it 
according with a continuity argument.
We fix $(q_\theta, r_{\theta})=(q,r)$ (where $q,r$ are given in Lemma \ref{AlgebricCauchy})
and we choose $\tilde q_\theta$ and $\tilde r_\theta$
as follows:
$$\frac 1{(\alpha+1) \tilde q_\theta'}= \frac{\theta}{q},
\quad \quad \frac 1{(\alpha+1) \tilde r_\theta '}=\frac \theta r + \frac{2(1-\theta)}{\alpha d}.$$
By this choice \eqref{eq.ass102v3ferro} and \eqref{eq.ass9bispferro} turn out
to be satisfied
for every $\theta$. On the other hand by \eqref{eq.ass9bisv}
we have
$$\lim_{\theta\rightarrow 1} \frac 1{\tilde q_\theta}=\frac 1{\tilde q},
\quad  \quad \lim_{\theta\rightarrow 1} \frac 1{\tilde r_\theta}=\frac 1{\tilde r},
$$
where $\tilde q, \tilde r$ are given by  Lemma \ref{AlgebricCauchy}.
Hence conditions \eqref{fewr} and \eqref{dropd12ferro} 
follow by \eqref{eq.ass3v3} and \eqref{dropd12} 
provided that we choose $\theta$ close enough to the value $\theta=1$.

\hfill$\Box$

The next lemma, which is a version of Lemma \ref{sTrbis} 
where we replace inequality by identity in the last condition, it will be useful in the sequel.

\begin{lem}\label{sTr}
Let $d\geq 1$ and $4/d<\alpha<4/(d-1)$ be fixed.
Then there exist $2<l\leq \infty, 2\leq p\leq \infty$
such that:
\begin{align}
\label{st1}&\frac 2l + \frac 1p=\frac 12,\\
\label{st2}&\frac 1{p'}=\frac 1p+\frac \alpha {r_\theta},\\
\label{st3}&\frac 1{l'}=\frac 1l+ \frac \alpha {q_\theta},
\end{align}
where $(q_\theta,r_\theta)$ is any couple given by Proposition
\ref{Algebric0}.
\end{lem}

The same proof as in Lemma~\ref{sTrbis}
can be repeated.  
 \section{Proof of Theorem \ref{main}}\label{thm}
 

 \begin{prop}\label{inpa}
 Let $(q_\theta,r_\theta)$ be as in 
 Proposition \ref{Algebric0} and $u(t,x,y)\in {\mathcal C}(\R;H^1_{x,y})$ be the unique global solution to \eqref{NLS}, with
 $4/d<\alpha<4/(d-1)$,
 then 
 \begin{equation}\label{set}u(t, x, y)\in L^{q_\theta}_t L^{r_\theta}_x H^{1/2+\delta}_y
 \end{equation}
 for some $\delta>0$.
  \end{prop}
 \begin{proof}
 We will apply a $H^{1/2+\delta}_y$ valued version of the analysis $H^s_x$ critical analysis of \cite{CW}.
 Notice that in Proposition \ref{Algebric0}
 we have $0<s<\frac 12$ and hence by choosing $\delta>0$ small ,
 we can control $\|.\|_{H^s_xH^{1/2+\delta}_y}$
 by $\|.\|_{H^1_{x,y}}$.
 By combining this fact with Remark~\ref{appl} we get
 \begin{align}
 \label{eq:202p90int}
\|u(t,x,y)\|_{L^{q_\theta}_{t>t_0} L^{r_\theta}_x H^{1/2+\delta}_y} 
\\\nonumber
\leq C\big (\|u(t_0)\|_{ H^1_{x,y}}+
\|u|u|^\alpha\|_{L^{\tilde{q_\theta}'}_{t>t_0} L^{\tilde{r}_\theta'}_xH_y^{1/2+\delta}}\big )
\\\nonumber
\leq C\big (\|u(t_0)\|_{ H^1_{x,y}}+
\|u\|_{L^{(1+\alpha)\tilde{q_\theta}'}_{t>t_0} L^{(1+\alpha)\tilde{r}_\theta'}_xH_y^{1/2+\delta}}^{1+\alpha}\big ),
\end{align}
where we used Lemma~\ref{TzBil}
and we have denoted by
$\|f(t)\|_{L^p_{t>t_0}}$ the integral $\int_{t_0}^\infty |f(t)|^p dt$ for any given
time-dependent function .
By combining \eqref{eq.ass9bispferro} with the H\"older inequality
we can continue the estimate \eqref{eq:202p90int} as follows: 
$$...\leq C\big(\|u(t_0)\|_{ H^1_{x,y}}+
\|u\|_{L^{q_\theta}_{t>t_0} L^{r_\theta}_xH_y^{1/2+\delta}}^{\theta(1+\alpha)}
\|u\|_{L^{\infty}_{t>t_0} L^{\alpha d/2}_xH_y^{1/2+\delta}}^{(1-\theta)(1+\alpha)}
\big ).
$$
By combining Proposition \ref{decay} with Lemma \ref{inter}, we deduce that
$$\lim_{t_0\rightarrow \infty}
\|u\|_{L^{\infty}_{t>t_0} L^{\alpha d/2}_xH_y^{1/2+\delta}}
=0,$$
and hence for every $\epsilon>0$ there exists $t_0=t_0(\epsilon)>0$ such that
$$\|u(t,x,y)\|_{L^{q_\theta}_{t>t_0} L^{r_\theta}_x H^{1/2+\delta}_y}
\leq C \|u(t_0)\|_{ H^1_{x,y}} +  \epsilon \|u\|_{L^{q_\theta}_{t>t_0} L^{r_\theta}_xH_y^{1/2+\delta}}^{\theta(1+\alpha)}.$$
We conclude by a continuity argument that 
$\|u(t,x,y)\|_{L^{q_\theta}_{t>0} L^{r_\theta}_x H^{1/2+\delta}_y}<\infty$.
By a similar argument we get 
$\|u(t,x,y)\|_{L^{q_\theta}_{t<0} L^{r_\theta}_x H^{1/2+\delta}_y}<\infty$.

 \end{proof}

 \begin{prop}\label{Impa}
 Let  $(l,p)$ be as in Lemma \ref{sTr}
 and let $u(t,x,y)$ be the unique solution to \eqref{NLS} with $4/d<\alpha<4/(d-1)$.
 Then 
 \begin{equation}\label{sett}\|u(t,x,y)\|_{L^l_t L^p_x L^2_y}
 + \|\partial_y u(t,x,y)\|_{L^l_t L^p_x L^2_y}
+\|\nabla_x u(t,x,y)\|_{L^l_t L^p_x L^2_y}
<\infty.\end{equation}
 \end{prop}

\begin{proof}
We show $\|\partial_y u(t,x,y)\|_{L^l_t L^p_x L^2_y}
<\infty$, the other estimates are similar.
By \eqref{eq:200p1} we get 
$$\|\partial_y u(t,x,y)\|_{L^l_{t>t_0} L^p_x L^2_y}
\leq C (\|u(t_0)\|_{H^1_{x,y}} + \|(\partial_y u)|u|^\alpha\|_{L^{l'}_{t>t_0} L^{p'}_x L^2_y}).$$
By Lemma \ref{sTr} we can apply the H\"older inequality and we get
$$...\leq C (\|u(t_0)\|_{H^1_{x,y}} + \|(\partial_y u)\|_{L^l_{t>t_0} L^p_x L^2_y} 
\|u\|_{L^{q_\theta}_{t>t_0} L^{r_\theta}_x L^\infty_y}^\alpha )
$$
$$\leq C (\|u(t_0)\|_{H^1_{x,y}} + \|(\partial_y u)\|_{L^l_{t>t_0} L^p_x L^2_y} 
\|u\|_{L^{q_\theta}_{t>t_0} L^{r_\theta}_x H^{1/2+\delta}}^\alpha ).
$$
We conclude by choosing $t_0$ large enough and by 
recalling Proposition \ref{inpa}.

\end{proof}

\begin{proof}[Proof of Theorem \ref{main}]
It follows by Proposition~\ref{Impa} via a standard argument (see \cite{Caz}).
In fact by using the integral equation associated with \eqref{NLS} it is sufficient to prove that
\begin{equation}\label{kdv}\lim_{t_1, t_2\rightarrow \infty}
\|\int_{t_1}^{t_2} e^{-i s\Delta_{x, y}} (u|u|^\alpha) ds\|_{H^1_{x,y}}=0
\end{equation}
By combining Proposition \ref{StriProd1} with a duality argument
we get:
$$\|\int_{t_1}^{t_2} e^{-i s\Delta_{x, y}} F(s) ds\|_{L^2_{x,y}}
\leq C \|F\|_{L^{l'}_{(t_1, t_2)} L^{p'}_x L^2_y}
$$
where $(l, p)$ are as in Lemma \ref{sTr}. Hence
\eqref{kdv} follows provided that 
$$
\lim_{t_1, t_2\rightarrow \infty} 
\big(\|u|u|^\alpha\|_{L^{l'}_{(t_1, t_2)} L^{p'}_x L^2_y}+
\|\partial_y (u|u|^\alpha)\|_{L^{l'}_{(t_1, t_2)} L^{p'}_x L^2_y}
+ \|\nabla_x (u|u|^\alpha)\|_{L^{l'}_{(t_1, t_2)} L^{p'}_x L^2_y}\big)=0.$$
This estimate can be proved following the same argument used along the proof
of Proposition \ref{Impa}, in conjunction with \eqref{set} and \eqref{sett}.
\end{proof}

\end{document}